\documentclass[12pt,reqno]{amsart}
\usepackage{amsmath,amsfonts,graphicx,amscd,amssymb,epsf,amsthm,enumerate,parskip,mathrsfs,color,ltablex,blindtext,cancel,cases}  
 \usepackage[a4paper, margin=1in]{geometry}
\usepackage{txfonts}
\usepackage{tikz-cd}
\usepackage{float}
\theoremstyle{definition}
\newtheorem{theorem}{Theorem}[section]
\newtheorem{proposition}[theorem]{Proposition}
\newtheorem{lemma}[theorem]{Lemma}
\newtheorem{corollary}[theorem]{Corollary}
\newtheorem{definition}[theorem]{Definition}

\newtheorem{question}[theorem]{Question}

\newtheorem{remark}[theorem]{Remark}
 \numberwithin{equation}{section}
\newcommand{\CC}{\mathbb{C}}

\newcommand{\ric}{\operatorname{Ric}}
\newcommand{\tr}{\operatorname{tr}}

\usepackage{cite}
\usepackage[colorlinks,citecolor=blue]{hyperref}

\begin{document}
  
\title{Rigidity of complete K\"ahler--Einstein metrics under cscK perturbations}
\author{Zehao Sha}
\address{Institut Fourier, UMR 5582, Laboratoire de Math\'ematiques,
Universit\'e Grenoble Alpes, CS 40700, 38058 Grenoble cedex 9, France}
\curraddr{C. N. Yang Institute for Mathematics and Fundamental Physics,
No. 1899, Shilian South Road, Hefei, Anhui, 230088, China}
\email{zhsha0251@gmail.com}
\begin{abstract}
In this paper, we study constant scalar curvature K\"ahler (cscK) metrics on complete non-compact K\"ahler--Einstein manifolds. We give sufficient conditions under which a cscK perturbation of a K\"ahler--Einstein metric must remain K\"ahler--Einstein. As a model case, we prove that the Bergman metric on a bounded strictly pseudoconvex domain is K\"ahler--Einstein whenever it has constant scalar curvature. In particular, combined with Huang--Xiao's resolution of Cheng's conjecture, this yields the ball characterization for smooth bounded strictly pseudoconvex domains.
\end{abstract}
%\subjclass[2020]{53A05}
%\keywords{K\"ahler--Einstein metric; CscK metric. }
\maketitle
%~~~~~~~~~~~~~~~~~~~~~~~~~~~~~~~~~~~~~~~~~~~~~~~~~~~~~~~~~~~~~~~~~~~~~~~~~~~~~~~~~~~~~~~~~~~~~~~~~~~~~~~~~~~~~~~~~~~~~~~~~~~~~~~~~~~~~~~~~~~~~~~~~~~~~~~~~~~~~~~
 
\section{Introduction}
\subsection{Canonical metrics}
A classical problem in K\"ahler geometry is to find a canonical representative in a given K\"ahler class $[\omega]$. K\"ahler–Einstein metrics are among the most natural and intensively studied ones. A K\"ahler metric $\omega$ is called K\"ahler--Einstein if the Ricci curvature is proportional to the metric, that is 
\begin{align*}
    \operatorname{Ric}(\omega) = \lambda \omega,
\end{align*}
for some $\lambda \in \mathbb{R}$. After rescaling the metric, one may normalize the Einstein constant to be $\lambda= 1,~0$ or $-1$. 

On compact K\"ahler manifolds, the existence of K\"ahler--Einstein metrics depends on the sign of the first Chern class $c_1(M)$:
\begin{enumerate}
    \item If $c_1(M)<0$, $M$ has an ample canonical bundle, a problem solved by Aubin \cite{aubin1976equations} and Yau \cite{yau1978ricci};
    \item If $c_1(M)=0$, $M$ is Calabi--Yau, as solved by Yau \cite{yau1978ricci};
    \item If \(c_{1}(M)>0\), \(M\) is Fano, and there are classical obstructions to the existence of K\"ahler--Einstein metrics, going back to Matsushima \cite{matsushima1957structure}, Futaki \cite{futaki1983obstruction}, and Tian \cite{tian1997kahler}. In a series of papers, Chen--Donaldson--Sun \cite{chen2015kahler1,chen2015kahler2,chen2015kahler3} proved that \emph{K}-polystability of a Fano manifold is sufficient for the existence of a K\"ahler--Einstein metric, confirming the ``\emph{K}-polystability \(\Rightarrow\) KE'' direction of the Yau--Tian--Donaldson conjecture.

\end{enumerate}
For the non-compact case, significant progress on the existence of the K\"ahler--Einstein metric has been made by Cheng--Yau \cite{Cheng1980OnTE}, Mok--Yau \cite{mokyau}, Tian--Yau \cite{tian1990complete,tian1991complete}, and Guedj--Kolev--Yeganefar \cite{guedj2013kahler}.

More generally, on a compact K\"ahler manifold, a K\"ahler--Einstein metric in a K\"ahler class $[\omega]$ necessarily satisfies $2\pi c_1(M)=\lambda[\omega]$ for some $\lambda\in\mathbb{R}$. When this cohomological condition fails, constant scalar curvature K\"ahler (often abbreviated as cscK) metrics provide the natural broader class of canonical metrics. On compact K\"ahler manifolds, the average of the scalar curvature $\hat{R}$ is given by
\begin{align*}
    \hat{R}= \frac{2n\pi c_1(M) \cup [\omega]^{n-1}}{[\omega]^n}
\end{align*}
which depends only on the K\"ahler class $[\omega]$. 

\subsection{YTD conjecture for the cscK metric}
For a polarized manifold $(M,L)$, the Yau--Tian--Donaldson conjecture predicts that the existence of a unique cscK metric in \(c_1(L)\) should be equivalent to the uniform \(K\)-stability of \((M, L)\), linking the $K$-energy's analytic behavior to algebraic stability via test configurations \cite{yau1992open,tian1997kahler,donaldson2002scalar}.

In \cite{chen2018existence}, Chen proposed a new continuity path for the cscK equation and proved the openness along this path. Subsequently, Chen and Cheng \cite{xxchen2021constant1,xxchen2021constant2} established the decisive a priori estimates and proved existence under the properness of the $K$-energy. We refer the reader to \cite{Stoppa09,berman2020regularity,berman2021variational,BHJ2019,BHJ2022,DarvasZhang2025,BJ2025YTD} for further developments.

%\begin{center}
%\begin{tikzcd}
%\text{Existence of the cscK metric in }c_1(L) \arrow[dd, "\text{~}"', shift right=3] & & \\
%& & \\
%\text{Properness of K-energy} \arrow[rr, "\text{~}"] \arrow[uu, "\text{~}"', shift right=3] & & {\text{K-stability of }(M,L)}
%\end{tikzcd}
%\end{center}

\subsection{CscK metrics on complete noncompact manifolds}
Despite the substantial progress in the compact case, cscK metrics on complete non-compact K\"ahler manifolds remain much less understood. A basic analytic difficulty is the failure of a global $\partial\bar\partial$-lemma in the non-compact setting, so one cannot in general work within a fixed K\"ahler class in the usual compact sense. For this reason, throughout the paper we fix a background complete K\"ahler metric $\omega$ and restrict attention to perturbations of the form $\omega_\varphi=\omega+\sqrt{-1}\,\partial\bar\partial\varphi$ with potential $\varphi$. Together with the (a priori unknown) log volume ratio $F$, the coupled cscK equation reads
\begin{equation}\label{cscKeqo}
  \begin{cases}
    \omega_\varphi^{n} = e^{F}\,\omega^{n},\\[2pt]
    \Delta_{\omega_\varphi} F = -\hat{R} + \operatorname{tr}_{\omega_\varphi}\operatorname{Ric}(\omega),
  \end{cases}
\end{equation}
whose solvability on suitable complete manifolds yields a cscK metric of scalar curvature $\hat{R}$. Note that K\"ahler--Einstein metrics also solve \eqref{cscKeqo}. On compact manifolds, if a cscK metric is cohomologous to a K\"ahler--Einstein metric, then it must itself be K\"ahler--Einstein. By contrast, in the complete non-compact case there is no general rigidity principle forcing cscK perturbations to coincide with the background K\"ahler--Einstein metric.

On a negatively curved K\"ahler--Einstein manifold $(M,\omega)$, we normalize by \(\ric(\omega)=-\omega\). For cscK perturbations with scalar curvature \(\hat{R}=-n\), the coupled system \eqref{cscKeqo} on $(M,\omega)$ reduces to
\[
\omega_\varphi^n=e^F\omega^n,
\qquad
\Delta_{\omega_\varphi}F=n-\tr_{\omega_\varphi}\omega,
\]
which is the form used throughout the paper.
\begin{question}\label{ques:main}
Can one give a \emph{sufficient condition} under which every cscK perturbation of a complete Kähler–Einstein metric is again Kähler–Einstein?
\end{question}

Question~\ref{ques:main} admits an affirmative answer under either of two sets of verifiable hypotheses: a parabolic condition, and a quantitative non-parabolic condition combining a positive spectral gap, a \(C^2\)-bound and an annular \(L^2\)-growth bound on the potential \(\varphi\). Specifically, our main result is the following.

\begin{theorem}\label{MT1}
Let $(M^n,\omega)$ be a complete K\"ahler--Einstein manifold with negative scalar curvature $R(\omega)=-n$. Suppose there exists \(\varphi\in C^\infty(M)\) normalized with \(\sup_M\varphi=0\) such that \(\omega_\varphi=\omega+\sqrt{-1}\partial\bar\partial\varphi\) defines a complete cscK metric satisfying \(R(\omega_\varphi)=-n\). Assume that \(\ric(\omega_\varphi)\) is bounded from below and that there exists a constant \(\lambda>0\) such that \(\omega_\varphi^n \ge \lambda\omega^n\). Then \(\omega_\varphi\) is K\"ahler--Einstein whenever one of the following conditions holds:
\begin{itemize}
    \item[(i)] \((M,\omega)\) is parabolic;

    \item[(ii)] \((M,\omega)\) is non-parabolic, \(\|\varphi\|_{C^2(M,\omega)}<\infty\) and there exist a point \(p\in M\), constants \(C_0>0\) and \(0<\delta<2\sqrt{\lambda_1(-\Delta_{\omega_\varphi})}\) such that, for all \(R\ge 1\),
\[
        \int_{B_{\omega_\varphi}(p,2R)
        \setminus B_{\omega_\varphi}(p,R)}
        |\varphi|^2\,\omega_\varphi^n
        \leq
        C_0e^{2\delta R}.
\]
\end{itemize}
\end{theorem}

\begin{remark}[Calabi--Yau case]
When the background metric is Calabi--Yau, decay of the logarithmic volume ratio at infinity already forces a scalar-flat perturbation to be Ricci-flat (see Proposition~\ref{CY case}).
\end{remark}

Recall that a complete Riemannian manifold is \emph{non-parabolic} if it admits a positive minimal Green's function, and \emph{parabolic} otherwise. In the present setting, Lemma~\ref{lem:vol_compa} implies that the normalized potential \(\varphi\) is \(\omega\)-subharmonic. This is sufficient to prove rigidity in the parabolic case. In the quantitative non-parabolic case, the natural analytic quantity is instead \(u:=F-\varphi\), which is \(\omega_\varphi\)-harmonic. Nevertheless, Lemma~\ref{lem:transfer-L2} shows that, under the \(C^2\)-bound on \(\varphi\), the annular \(L^2\)-growth conditions for \(u\) and \(\varphi\) are equivalent. We therefore formulate Theorem~\ref{MT1} (ii) directly in terms of the potential \(\varphi\). Moreover, Lemma~\ref{lem:metric-equivalence} shows that \(\omega_\varphi\) and \(\omega\) are uniformly equivalent. Consequently, \(\lambda_1(-\Delta_{\omega_\varphi})>0\) if and only if \(\lambda_1(-\Delta_\omega)>0\). The precise growth threshold in Theorem~\ref{MT1} (ii) is written using \(\omega_\varphi\), since both the numerical value of \(\lambda_1\) and the radial variable in the annular condition depend on the chosen metric.
%~~~~~~
\subsection{Bounded strictly pseudoconvex domains as a model case}

As a concrete model case of Question~\ref{ques:main}, we next consider bounded strictly pseudoconvex domains, where the rigidity problem admits a particularly geometric form. Let $\Omega \subset \CC^n$ be a bounded strictly pseudoconvex domain with smooth boundary. In \cite{Cheng1980OnTE}, Cheng and Yau proved that there is a unique complete Kähler-Einstein metric on $\Omega$ constructed by a global strictly plurisubharmonic defining function. This leads to the following natural question: \emph{Can a bounded strictly pseudoconvex domain carry a complete cscK metric that is not K\"ahler--Einstein?}

On a bounded strictly pseudoconvex domain, two natural types of complete K\"ahler metrics arise. The first is the Bergman metric $\omega_B$. The second consists of the metrics $\omega_\rho$ associated with strictly plurisubharmonic defining functions \(\rho\). Any two metrics of the latter type differ by a global $\sqrt{-1}\,\partial\bar\partial$-term. We refer to Section~\ref{sec:pcv} for the relevant background and details. For these two types of metrics, we prove the following rigidity statements.

\begin{theorem}\label{MT2}
Let $\Omega \subset \mathbb{C}^n$ be a bounded strictly pseudoconvex domain with $C^2$-boundary.
\begin{enumerate}
\item If $\omega_\rho$ is cscK and $\rho\in C^\infty(\Omega)\cap C^4(\bar{\Omega})$ satisfies \(J(\rho)=1\) on \(\partial \Omega\) where \(J\) is the Fefferman operator, then $\omega_\rho$ coincides with the unique K\"ahler--Einstein metric constructed by Cheng--Yau.
  \item If the Bergman metric $\omega_B$ is cscK, then $\omega_B$ is K\"ahler–Einstein. Moreover, if $\partial \Omega \in C^\infty$, then $\Omega$ is biholomorphic to the unit ball.
\end{enumerate}
\end{theorem}

Theorem~\ref{MT2} (2) may be viewed as a cscK extension of Huang--Xiao's rigidity theorem for K\"ahler--Einstein Bergman metrics \cite[Theorem~1]{huang2021bergman}. Indeed, once the Bergman metric is assumed to have constant scalar curvature, our rigidity result forces it to be K\"ahler--Einstein. Then Huang--Xiao's result yields the ball characterization on bounded smooth strictly pseudoconvex domains.

\subsection*{Acknowledgment}
The author would like to express his sincere gratitude to his advisors, Gérard Besson and Hervé Gaussier, for their constant encouragement, guidance, and helpful discussions. The author is also grateful to the referees for their valuable comments and suggestions.
%~~~~~~~~~~~~~~~~~~~~~~~~~~~~~~~~~~~~~~~~~~~~~~~~~~~~~~~~~~~~~~~~~~~~~~~~~~~~~~~~~~~~~~~~~~~~~~~~~~~~~~~~~~~~~~~~~~~~~~~~~~~~~~~~~~~~~~~~~~~~~~~~~~~~~~~~~~~~~~~
\section{Preliminaries}

\subsection{Notations}

Let \((M^n,\omega)\) be an \(n\)-dimensional complete K\"ahler manifold. In local holomorphic coordinates \((z^1,\ldots,z^n)\), we write
\[
\omega=\sqrt{-1}\,g_{i\bar j}\,dz^i\wedge d\overline{z}^j,
\]
where we adopt the Einstein summation convention. We use \(\partial_i\) and \(\partial_{\bar j}\) as shorthand for \(\frac{\partial}{\partial z^i}\) and \(\frac{\partial}{\partial \overline{z}^j}\), respectively, and denote by \((g^{i\bar j})\) the inverse matrix of \((g_{i\bar j})\). Given two tensors \(A\) and \(B\) of the same type, their Hermitian inner product with respect to \(\omega\) is denoted by \(\langle A,B\rangle_\omega\), and the corresponding norm by \(|A|_\omega\).

The Ricci form is given by
\[
\ric(\omega)=\sqrt{-1}\,R_{i\bar j}\,dz^i\wedge d\overline{z}^j,
\qquad
R_{i\bar j}
=
-\partial_i\partial_{\bar j}\log\det(g_{k\bar l}),
\]
and the scalar curvature is
\[
R(\omega)=\tr_\omega\ric(\omega)=g^{i\bar j}R_{i\bar j}.
\]

For \(f\in C^\infty(M,\mathbb R)\), the gradient vector field is
\[
\nabla f
=
\frac{1}{2}g^{i\bar j}\bigl(\partial_i f\,\partial_{\bar j}+\partial_{\bar j}f\,\partial_i\bigr),
\]
and the Laplacian is defined by the trace of the complex Hessian $\sqrt{-1}\partial\bar\partial f$ with respect to $\omega$, namely
\[
\Delta_\omega f
:=
\tr_\omega\bigl(\sqrt{-1}\partial\bar\partial f\bigr)
=
g^{i\bar j}\partial_i\partial_{\bar j}f.
\]

\subsection{Omori--Yau maximum principle}

We record the generalized maximum principle on complete K\"ahler manifolds in the two forms used later in the paper.

\begin{proposition}[Omori \cite{omori1967isometric}, Yau \cite{yau1975harmonic}]
Let \((M,\omega)\) be a complete K\"ahler manifold, and let \(u\in C^2(M)\) satisfy \(\sup_M u<+\infty\).

\begin{itemize}
    \item[(i)] If the sectional curvature of \(\omega\) is bounded from below, then there exists a sequence \(\{z_k\}\subset M\) such that
    \[
    \lim_{k\to\infty}u(z_k)=\sup_M u,
    \qquad
    \lim_{k\to\infty}|\nabla u(z_k)|_\omega=0,
    \qquad
    \limsup_{k\to\infty}\sqrt{-1}\partial\bar\partial u(z_k)\le 0,
    \]
    where the last inequality is understood in the sense of Hermitian forms.

    \item[(ii)] If the Ricci curvature of \(\omega\) is bounded from below, then there exists a sequence \(\{z_k\}\subset M\) such that
    \[
    \lim_{k\to\infty}u(z_k)=\sup_M u,
    \qquad
    \lim_{k\to\infty}|\nabla u(z_k)|_\omega=0,
    \qquad
    \limsup_{k\to\infty}\Delta_\omega u(z_k)\le 0.
    \]
\end{itemize}
\end{proposition}

%~~~~~~~~~~~~~~~~~~~~~~~~~~~~~~~~~~~~~~~~~~~~~~~~~~~~~~~~~~~~~~~~~~~~~~~~~~~~~~~~~~~~~~~~~~~~~~~~~~~~~~~~~~~~~~~~~~~~~~~~~~~~~~~~~~~~~~~~~~~~~~~~~~~~~~~~~~~~~~~
\section{Basic estimates for cscK metrics}
Let \((M^n,\omega)\) be a complete K\"ahler--Einstein manifold with negative scalar curvature, normalized by \(R(\omega)=-n\). Suppose \(\tilde{\omega}\) is a complete cscK metric on \(M\) with
\[
R(\tilde{\omega})=R(\omega)=-n.
\]
Then there exists a smooth function \(F\) such that
\begin{equation}\label{csck eq:vol}
    \tilde{\omega}^n=e^F \omega^n,
\end{equation}
and 
\begin{equation}\label{csck eq:F}
    \Delta_{\tilde{\omega}} F =n-\operatorname{tr}_{\tilde{\omega}}\omega.
\end{equation}
We can verify that 
\begin{align*}
    R(\tilde{\omega})&=\operatorname{tr}_{\tilde{\omega}} \operatorname{Ric}(\tilde{\omega})\\
    &=-\operatorname{tr}_{\tilde{\omega}} \sqrt{-1}\partial\bar{\partial}\log\left(\frac{\tilde{\omega}^n}{\omega^n}\right)+\operatorname{tr}_{\tilde{\omega}} \operatorname{Ric}(\omega)\\
    &=-\Delta_{\tilde{\omega}}F -\operatorname{tr}_{\tilde{\omega}}\omega\\
    &=-n.
\end{align*}
Throughout this section, we work under the above normalization.

We first present a basic comparison for the volume form of a cscK metric.
\begin{lemma} \label{lem:vol_compa}
Let \((M,\omega)\) be a complete K\"ahler--Einstein manifold, and let \(\tilde{\omega}\) be a complete cscK metric with \(R(\tilde{\omega})=R(\omega)=-n\). If \(\ric(\tilde{\omega})\) is bounded from below and the log volume ratio \(F\) is bounded from below, then
    \begin{align} \label{comparison vol 1}
        \tilde{\omega}^n \geq \omega^n.
    \end{align}
    Moreover, if equality holds at any point $p \in M$, then $\tilde{\omega} = \omega$ on $M$.
\end{lemma}
\begin{proof}
Writing $u=\omega^n / \tilde{\omega}^n  >0$, we compute
\begin{align*}
    \Delta_{\tilde{\omega}} \log u &= \Delta_{\tilde{\omega}} \log \left(\omega^n/\tilde{\omega}^n\right)\\
    &= \operatorname{tr}_{\tilde{\omega}}\omega -n \\
    & \geq n u^{\frac{1}{n}} - n,
\end{align*}
where the last inequality follows from the arithmetic-geometric inequality. 

Then we obtain
\begin{align*}
    \Delta_{\tilde{\omega}} \log u = \frac{\Delta_{\tilde{\omega}} u}{ u } - \frac{|\nabla u|^2_{\tilde{\omega}}}{u^2}
    \geq n u^{\frac{1}{n}} - n,
\end{align*}
which leads to
\begin{align*}
    \Delta_{\tilde{\omega}} u - \frac{|\nabla u|^2_{\tilde{\omega}}}{u}\geq n u^{\frac{1}{n}+1} - n u.
\end{align*}

Since \(u\) is bounded from above and \(\ric(\tilde{\omega})\) is bounded from below, the generalized maximum principle applied to \(u\) with respect to \(\tilde{\omega}\) yields a sequence \(\{z_\alpha\}\subset M\) such that
\begin{align*}
    0 \geq \limsup_{\alpha \rightarrow \infty}\Delta_{\tilde{\omega}} u(z_\alpha) \geq n \lim_{\alpha \rightarrow \infty}u^{\frac{1}{n}+1}(z_\alpha) - n\lim_{\alpha \rightarrow \infty} u(z_\alpha).
\end{align*}
This implies $ \sup_M u \leq 1$, yielding the desired inequality.

Observe that
\begin{align*}
    \Delta_{\tilde{\omega}} \log u \ge n u^{\frac{1}{n}} - n 
    \ge  \log u.
\end{align*}
Suppose there is a point $p \in M$ such that $u(p)=\sup_M u =1$. By the strong maximum principle, we have $\log u\equiv 0$ in $M$ and $\tilde{\omega}^n =\omega^n$. Then \eqref{csck eq:F} yields $\operatorname{tr}_{\tilde{\omega}}\omega= n$. Consequently, we have $\tilde{\omega} =\omega$ thanks to the equality case of the arithmetic-geometric inequality.
\end{proof}

\begin{remark}
Under the assumptions of Lemma~\ref{lem:vol_compa}, we have
\[
F=\log\frac{\tilde\omega^n}{\omega^n}\ge 0.
\]
Moreover, if \(F(p)=0\) at a point \(p\in M\), then \(\tilde\omega^n=\omega^n\) at \(p\), hence \(\tilde\omega=\omega\) on \(M\).
\end{remark}

\begin{lemma}\label{lem:metric-equivalence}
Under the assumptions of Lemma~\ref{lem:vol_compa}, suppose in addition that \(\tilde\omega= \omega+\sqrt{-1}\partial\bar\partial\varphi\) for some \(\varphi\in C^\infty(M)\) satisfying
\(\|\varphi\|_{C^2(M,\omega)}\le A\). Then there exists a constant \(C=C(n,A)>1\) such that
\[
C^{-1}\omega
\le
\tilde\omega
\le
C\omega.
\]
\end{lemma}

\begin{proof}
Fix a point \(x\in M\), and choose \(\omega\)-normal coordinates in which \(\tilde\omega\) is diagonal. Let \(\lambda_1,\ldots,\lambda_n\) be the eigenvalues of \(\tilde\omega\) relative to \(\omega\). The \(C^2\)-bound on \(\varphi\) gives
\[
\lambda_j\le \Lambda
\qquad
\text{for all }j,
\]
where \(\Lambda>1\) depends only on \(A\). On the other hand, Lemma~\ref{lem:vol_compa} gives
\[
\prod_{j=1}^n\lambda_j
=
\frac{\tilde\omega^n}{\omega^n}
\ge1.
\]
Hence, for every \(j\),
\[
\lambda_j
=
\frac{\prod_{k=1}^n\lambda_k}
{\prod_{k\ne j}\lambda_k}
\ge
\Lambda^{-(n-1)}.
\]
This proves the metric comparison.
\end{proof}
%~~~~~~~~~~~~~~~~~~~~~~~~~~~~~~~~~~~~~~~~~~~~~~~~~~~~~~~~~~~~~~~~~~~~~~~~~~~~~~~~~~~~~~~~~
\section{CscK metrics on a complete K\"ahler--Einstein manifold}\label{sec:main-sec}
\subsection{Asymptotic rigidity and motivating examples}
Let \((M,\omega)\) be a complete K\"ahler--Einstein manifold with scalar curvature \(-n\), and let
\[
\omega_\varphi:=\omega+\sqrt{-1}\partial\bar\partial\varphi
\]
be a complete cscK perturbation of \(\omega\), where \(\varphi\in C^\infty(M)\). Suppose
\begin{equation}\label{eq:cscK}
\begin{cases}
    \omega_\varphi^n=e^F\omega^n,\\
    \Delta_{\omega_\varphi}F=n-\tr_{\omega_\varphi}\omega,
\end{cases}
\end{equation}
for some \(F\in C^\infty(M)\). Without loss of generality, by adding a constant to \(\varphi\), we normalize \(\sup_M\varphi=0\). We will keep this convention throughout this section.

On compact K\"ahler manifolds, cscK metrics cohomologous to a K\"ahler--Einstein metric are automatically K\"ahler--Einstein:

\begin{proposition}[A well-known fact]\label{cscKisKE}
Let \((M,\omega)\) be a compact K\"ahler manifold. Suppose that \(2\pi c_1(M)=\lambda[\omega]\) for some constant \(\lambda\), where \(c_1(M)\) is the first Chern class of \(M\). If \(\omega\) is a cscK metric, then \(\omega\) is K\"ahler--Einstein.
\end{proposition}

Question~\ref{ques:main} asks for a non-compact analogue of this rigidity statement. In the complete non-compact setting, such a conclusion can no longer follow from compact cohomological arguments alone, so one must impose additional information. Accordingly, a natural strategy is to seek sufficient conditions formulated at infinity. Relative to the reference K\"ahler--Einstein metric \(\omega\), the pair \((F,\varphi)\) satisfies the coupled elliptic system \eqref{eq:cscK}. This makes asymptotic information particularly relevant: analytic behavior at infinity is precisely the type of information that can be translated into the interior through the elliptic maximum principle. From this perspective, it is natural to begin with asymptotic conditions as candidates for rigidity.

For this purpose, we distinguish two asymptotic behaviors for a cscK perturbation \(\omega_\varphi\):
\begin{enumerate}
    \item (\textit{Weak asymptotic condition}) the volume ratio satisfies
    \[
    \frac{\omega_\varphi^n}{\omega^n}\to 1
    \qquad \text{at infinity};
    \]
    \item (\textit{Strong asymptotic condition}) for any \(p \in M\), the metric perturbation satisfies
\[
\lim_{r\to\infty}
\|\varphi\|_{C^2(M\setminus B_\omega(p,r),\omega)}
=
0.
\]
\end{enumerate}
The strong condition clearly implies the weak one.

Since a complete Einstein manifold with positive Einstein constant is compact by Myers' theorem, the non-compact K\"ahler--Einstein setting reduces to two cases: the Calabi--Yau case and the negatively curved case. We begin with the Calabi--Yau case, where rigidity is particularly transparent: weak asymptotic control of the volume ratio already suffices.
\begin{proposition}\label{CY case}
Let $(M,\omega)$ be a complete Calabi--Yau manifold. Suppose $\tilde{\omega}$ is a cscK metric on $M$ with $R(\tilde{\omega})=0$, and that $\tilde{\omega}^{n}=e^{F}\omega^{n}$ for some $F\in C^{\infty}(M)$. If, for any point $p\in M$,
\[
  \lim_{r\to\infty}\|F\|_{C^{0}(M\setminus B(p,r))}=0,
\]
then $\tilde{\omega}$ is Ricci-flat.
\end{proposition}

\begin{proof}
Since $\operatorname{Ric}(\tilde{\omega})=\operatorname{Ric}(\omega)-\sqrt{-1}\,\partial\bar\partial F$ and $\operatorname{Ric}(\omega)=0$, tracing with respect to $\tilde{\omega}$ gives
\[
  0=R(\tilde{\omega})=\operatorname{tr}_{\tilde{\omega}}\operatorname{Ric}(\tilde{\omega})
   = -\Delta_{\tilde{\omega}}F,
\]
which implies $F$ is harmonic. As $F\to 0$ at infinity, the maximum principle yields $F\equiv 0$, hence $\tilde{\omega}^{n}=\omega^{n}$ and $\operatorname{Ric}(\tilde{\omega})=\operatorname{Ric}(\omega)=0$.
\end{proof}

We now turn to the negatively curved case. The cscK equation with \(R(\omega_\varphi)=R(\omega)\) already implies that \(F-\varphi\) is \(\omega_\varphi\)-harmonic. As a first rigidity result, we impose a stronger asymptotic condition which forces this harmonic function to vanish at infinity.
\begin{proposition}\label{cscK is KE 2}
Let \((M,\omega)\) be a complete K\"ahler--Einstein manifold with negative scalar curvature. Suppose there exists \(\varphi\in C^\infty(M)\) such that \(\omega_\varphi=\omega+\sqrt{-1}\partial\bar\partial\varphi\) defines a complete cscK metric with \(R(\omega_\varphi)=R(\omega)=-n\). If, for every \(p\in M\),
\[
\lim_{r\to\infty}\|\varphi\|_{C^2(M\setminus B(p,r))}=0,
\]
then \(\varphi\equiv 0\) and \(\omega_\varphi=\omega\).
\end{proposition}
\begin{proof}
It follows from \eqref{eq:cscK} that
\[
\Delta_{\omega_\varphi}(F-\varphi)=0.
\]
Since \(\lim_{r\to\infty}\|\varphi\|_{C^2(M\setminus B(p,r))}=0\) for any \(p\in M\), \(\varphi\to0\) at infinity. Moreover, the first equation in \eqref{eq:cscK} gives
\[
F
=
\log\det\left(
I+\omega^{-1}\sqrt{-1}\partial\bar\partial\varphi
\right)
\longrightarrow0.
\]
Hence \(F-\varphi\to0\) at infinity. The maximum principle then implies that \(F-\varphi\equiv0\) on \(M\). Substituting this into (\ref{eq:cscK}) yields  
\begin{equation}\label{eq:det-relation}
       (\omega + \sqrt{-1}\partial \bar{\partial}\varphi)^n = e^{\varphi} \omega^n.
\end{equation}  
Since \(\varphi\to 0\) at infinity, the function \(\varphi\) is bounded on \(M\). If \(\sup_M \varphi >0\), then $\varphi$ attains a positive maximum at an interior point. At that point, we have
\[
1  \geq \frac{(\omega + \sqrt{-1}\partial \bar{\partial}\varphi)^n}{\omega^n} = e^{\sup_{M} \varphi},
\]  
which is a contradiction. This implies \(\sup_{M} \varphi \leq 0\). Similarly, an interior negative minimum would give
\[
e^{\inf_M\varphi}\geq1,
\]
again a contradiction. Thus $\varphi\equiv0$, and consequently
$\omega_\varphi=\omega$.
\end{proof}

%~~~~~~
\subsection{Parabolicity and rigidity of complete K\"ahler--Einstein metrics}

In particular, Proposition~\ref{cscK is KE 2} shows that if a K\"ahler--Einstein metric and a cscK metric are asymptotically equivalent in the strong sense, then they must coincide. However, the condition
\[
\|\varphi\|_{C^2}\longrightarrow 0
\]
already yields \(C^0\)-decay of \(\varphi\) and \(F\), and hence provides explicit $C^0$-control at infinity. From the analytic point of view, it is therefore natural to ask whether this asymptotic \(C^0\)-information of $F$ and $\varphi$ can be weakened to a more flexible integral condition, preferably of \(L^2\)-type. Once we move from pointwise decay to global integral control, the distinction between parabolic and non-parabolic manifolds becomes essential.
\begin{definition}[Parabolicity]
A complete Riemannian manifold \((M,g)\) is called \emph{non-parabolic} if it admits a positive minimal Green's function. Otherwise, it is called \emph{parabolic}.
\end{definition}

A basic consequence of parabolicity is that every subharmonic function bounded from above must be constant (see, for instance, \cite{grigor1999analytic}). In our setting, this applies directly to the perturbation potential \(\varphi\) when the background K\"ahler--Einstein manifold is parabolic. It follows from Lemma \ref{lem:vol_compa} that 
\begin{align}\label{subharmonic-phi}
    \Delta_\omega \varphi  = \tr_\omega \omega_{\varphi} - n \geq n\left(\omega_\varphi^n/\omega^n\right)^{1/n} - n \geq 0.
\end{align}
Hence we obtain the following rigidity statement in the parabolic case.
\begin{proposition}\label{cscK is KE:para}
Under the same assumptions as in Theorem \ref{MT1}, if \((M,\omega)\) is parabolic, then \(\varphi\equiv 0\). In particular, \(\omega_\varphi\) is K\"ahler--Einstein.
\end{proposition}
\begin{remark}[Quadratic volume growth implies parabolicity]
A standard volume-growth criterion asserts that if
\[
\int^\infty \frac{r}{\operatorname{Vol}_g (B(p,r))}\,dr=\infty,
\]
then \((M,g)\) is parabolic. In particular, if \(\operatorname{Vol}_g( B(p,r))\le C r^2\) for all sufficiently large \(r\), then the integral diverges and \((M,g)\) is parabolic; see also \cite{cheng1975differential}.
\end{remark}

We now turn to the non-parabolic case. For the rigidity argument below, we impose the stronger quantitative assumption that
\[
\lambda_1(-\Delta_\omega):= \inf_{0\neq f \in W^{1,2}_0(M)} \frac{\int_M |\nabla f|_\omega^2 ~\omega^n}{\int_M |f|^2~\omega^n}>0.
\]
This implies \((M,\omega)\) is non-parabolic. Moreover, under the additional assumption \(\|\varphi\|_{C^2(M,\omega)}<\infty\), the perturbation metric \(\omega_\varphi\) is uniformly equivalent to the reference metric \(\omega\) thanks to Lemma \ref{lem:metric-equivalence}. We first show the following Liouville-type statement.

\begin{lemma}\label{lem:annular-L2-Liouville}
Let \((M,g)\) be a complete K\"ahler manifold satisfying \(\lambda_1(-\Delta_g)>0\). Let \(u\in C^\infty(M)\) be \(g\)-harmonic. Suppose that, for some point \(p\in M\), there exist constants \(C_0>0\) and \(0<\delta<2\sqrt{\lambda_1(-\Delta_g)}\) such that for every \(R \ge 1\), we have
\begin{equation}\label{eq:L2-growth}
\int_{B_g(p,2R)\setminus B_g(p,R)}
u^2\,dV_g
\le
C_0e^{2\delta R}.
\end{equation}
Then \(u\equiv0\).
\end{lemma}
\begin{proof}
Fix a number \(\alpha\in\bigl(\tfrac\delta2,\sqrt{\lambda_1(-\Delta_g)}\bigr)\), and let \(\rho(x):=d_g(p,x)\) be the distance function. Choose a nonincreasing function $\chi\in C^\infty([0,\infty))$ such that
\[
0\le \chi\le 1,\qquad
\chi\equiv 1 \ \text{on } [0,1],\qquad
\chi\equiv 0 \ \text{on } [2,\infty),\qquad
|\chi'|\le C.
\]
For $R\ge 1$, define
\[
\chi_R(x):=\chi\!\left(\frac{\rho(x)}{R}\right),
\qquad
\eta_R(x):=\chi_R(x)e^{-\alpha \rho(x)}.
\]
Then $\eta_R\in W^{1,\infty}_c(M)$, $\operatorname{supp} \eta_R\subset B_g(p,2R)$, and
\[
\eta_R=e^{-\alpha\rho}\quad\text{on } B_g(p,R).
\]
In particular,
\[
\eta_Ru,\ \eta_R^2u\in W^{1,2}_0(M).
\]

Since $u$ is $g$-harmonic, its weak formulation gives
\[
\int_M \langle \nabla u,\nabla \psi\rangle_g\,dV_g=0
\qquad\forall\,\psi\in W^{1,2}_0(M).
\]
Taking $\psi=\eta_R^2u$, we obtain
\[
0=\int_M \eta_R^2 |\nabla u|_g^2\,dV_g
+2\int_M \eta_R u\langle \nabla u,\nabla\eta_R\rangle_g\,dV_g.
\]
On the other hand,
\[
|\nabla(\eta_R u)|_g^2
=
\eta_R^2|\nabla u|_g^2
+u^2|\nabla\eta_R|_g^2
+2\eta_R u\langle \nabla u,\nabla\eta_R\rangle_g.
\]
Combining the last two identities yields
\[
\int_M |\nabla(\eta_R u)|_g^2\,dV_g
=
\int_M u^2|\nabla\eta_R|_g^2\,dV_g.
\]
By the variational characterization of $\lambda_1(-\Delta_g)$,
\[
\lambda_1(-\Delta_g)\int_M (\eta_R u)^2\,dV_g
\le
\int_M |\nabla(\eta_R u)|_g^2\,dV_g,
\]
hence
\[
\lambda_1(-\Delta_g)\int_M \eta_R^2 u^2\,dV_g
\le
\int_M u^2|\nabla\eta_R|_g^2\,dV_g.
\]
Since
\[
\nabla\eta_R
=
e^{-\alpha\rho}\nabla\chi_R
-
\alpha\chi_Re^{-\alpha\rho}\nabla\rho
\]
and $|\nabla\rho|_g\le 1$ almost everywhere, for every $\varepsilon>0$ we have
\[
|\nabla\eta_R|_g^2
\le
(1+\varepsilon)\alpha^2\eta_R^2
+
C_\varepsilon e^{-2\alpha\rho}|\nabla\chi_R|_g^2.
\]
Since $|\nabla\chi_R|_g\le C/R$ and $\nabla\chi_R$ is supported in \(A_R:=B_g(p,2R)\setminus B_g(p,R)\), we get
\[
|\nabla\eta_R|_g^2
\le
(1+\varepsilon)\alpha^2\eta_R^2
+
\frac{C_\varepsilon}{R^2}e^{-2\alpha\rho}\mathbf 1_{A_R}.
\]
Choose $\varepsilon>0$ so small that \(\lambda_1(-\Delta_g)-(1+\varepsilon)\alpha^2>0\). Therefore,
\[
\Bigl(\lambda_1(-\Delta_g)-(1+\varepsilon)\alpha^2\Bigr)
\int_M \eta_R^2 u^2\,dV_g
\le
\frac{C_\varepsilon}{R^2}
\int_{A_R} e^{-2\alpha\rho}u^2\,dV_g.
\]

Let \(E_k:=B_g(p,k+1)\setminus B_g(p,k)\) for \(k\geq1\).
Since
\[
E_k
\subset
B_g(p,k+1)\setminus
B_g\left(p,\frac{k+1}{2}\right)
=A_{(k+1)/2},
\]
we have
\[
\int_{E_k}u^2\,dV_g
\leq
C_0e^{\delta(k+1)}.
\]
It follows from \(A_R\subset\bigcup_{k=\lfloor R\rfloor}^{\lceil2R\rceil}E_k\) and \(\rho\geq k\) on \(E_k\) that
\begin{align*}
\int_{A_R}e^{-2\alpha\rho}u^2\,dV_g
&\leq
\sum_{k=\lfloor R\rfloor}^{\lceil2R\rceil}
e^{-2\alpha k}
\int_{E_k}u^2\,dV_g\\
&\leq
C_0e^\delta
\sum_{k=\lfloor R\rfloor}^{\lceil2R\rceil}
e^{-(2\alpha-\delta)k}\\
&\leq
C_{\alpha,\delta}
e^{-(2\alpha-\delta)R}.
\end{align*}
Consequently,
\[
\Bigl(
\lambda_1(-\Delta_g)-(1+\varepsilon)\alpha^2
\Bigr)
\int_M\eta_R^2u^2\,dV_g
\leq
\frac{C}{R^2}
e^{-(2\alpha-\delta)R}
\xrightarrow[R\to\infty]{}0.
\]
Since \(\chi\) is nonincreasing, for each fixed \(x\in M\) we have
\[
\eta_R(x)^2u(x)^2 \nearrow e^{-2\alpha\rho(x)}u(x)^2
\qquad\text{as }R\to\infty.
\]
By the monotone convergence theorem,
\[
\int_M e^{-2\alpha\rho}u^2\,dV_g
=
\lim_{R\to\infty}\int_M \eta_R^2u^2\,dV_g
=0.
\]
Hence $u\equiv 0$ on $M$.
\end{proof}

\begin{lemma} \label{lem:transfer-L2}
Under the same assumption as in Theorem \ref{MT1}. Suppose in addition that \(\|\varphi\|_{C^2(M,\omega)} < \infty\). Then there exist constants $C>0$ and $L>1$, independent of $p\in M$ and $R$, such that, for every $R>2L$,
\begin{equation}\label{eq:annular-transfer}
\begin{split}
 \int_{B_{\omega_\varphi}(p,2R)\setminus B_{\omega_\varphi}(p,R)}
        |F-\varphi|^2\,\omega_\varphi^n                                     
\leq
 C\int_{B_{\omega_\varphi}(p,2R+L)\setminus B_{\omega_\varphi}(p,R-L)}
        |\varphi|^2\,\omega_\varphi^n .
\end{split}
\end{equation}
Moreover, for every \(p\in M\) and every \(\delta\geq0\),
\begin{equation}\label{eq:annular-growth-equivalence}
\int_{B_{\omega_\varphi}(p,2R)\setminus B_{\omega_\varphi}(p,R)}
|\varphi|^2\,\omega_\varphi^n
=
O\left(e^{2\delta R}\right)
\quad\Longleftrightarrow\quad
\int_{B_{\omega_\varphi}(p,2R)\setminus B_{\omega_\varphi}(p,R)}
|F-\varphi|^2\,\omega_\varphi^n
=
O\left(e^{2\delta R}\right)
\end{equation}
as \(R\to\infty\).
\end{lemma}

\begin{proof}
Write again \(u:=F-\varphi,~g:=\omega_\varphi,~dV_g:=\omega_\varphi^n\). We first prove a local estimate. At any point, choose $\omega$-normal coordinates in which $g$ has eigenvalues $\lambda_1,\ldots,\lambda_n$ relative to $\omega$. Since $\log t\leq t-1$ for every $t>0$,
\[
        F=
        \sum_{j=1}^n\log\lambda_j                        
        \leq
        \sum_{j=1}^n(\lambda_j-1)
        =
        \Delta_\omega\varphi.
\]
It follows from Lemma \ref{lem:vol_compa} and \(\varphi \le 0\) that
\begin{equation}\label{eq:u_deltaphi}
    0 \le u \le \Delta_\omega\varphi -\varphi.
\end{equation}

Fix \(x\in M\). By Lemma~\ref{lem:metric-equivalence}, there exists a constant \(A\geq1\), independent of \(x\), such that
\[
A^{-1}\omega\leq g\leq A\omega.
\]
Let \(r_x=d_\omega(x,\cdot)\), and choose a non-increasing function \(\chi\in C^\infty([0,\infty))\) such that
\[
0\leq\chi\leq1,
\qquad
\chi\equiv1\ \text{on }[0,1],
\qquad
\chi\equiv0\ \text{on }[2,\infty).
\]
Define \(\eta_x:=\chi\bigl(A^{-1/2}r_x\bigr)\). Then \(\eta_x\equiv1\) on \(B_g(x,1)\) and \(\operatorname{supp}\eta_x \subset B_g(x,L)\), where we may take \(L=3A\). Since \(\ric(\omega)=-\omega\), Laplacian comparison gives a uniform upper bound for \(\Delta_\omega r_x\) on the region \(A^{1/2}\leq r_x\leq2A^{1/2}\). Because \(\chi'\leq0\), it follows that \(\Delta_\omega\eta_x\geq-C_\eta\) in the distributional sense, where \(C_\eta\) is independent of \(x\). Using \eqref{eq:u_deltaphi} and \(\varphi\leq0\), we obtain
\[
\begin{aligned}
\int_{B_g(x,1)}u\,dV_g
&\leq
C\int_M\eta_xu\,\omega^n\\
&\leq
C\int_M
\eta_x\bigl(\Delta_\omega\varphi-\varphi\bigr)\,\omega^n\\
&=
C\left\langle\Delta_\omega\eta_x,\varphi\right\rangle_\omega
-
C\int_M\eta_x\varphi\,\omega^n\\
&\leq
C\int_{B_g(x,L)}|\varphi|\,dV_g.
\end{aligned}
\]
Here the third line is understood in the distributional sense, and the last inequality follows precisely from \(\Delta_\omega\eta_x\geq-C_\eta\) and \(\varphi\leq0\).

If \(u\equiv0\), there is nothing to prove. Otherwise, since \(u\geq0\) and \(\Delta_g u=0\), the strong maximum principle implies that \(u>0\) on \(M\). Recall Harnack's inequality \cite[Corollary 3.3]{SYbook},
\begin{equation}\label{eq:Harnack}
\sup_{B_g(x,1)}u
\leq C\inf_{B_g(x,1)}u,
\end{equation}
where \(C\) depends only on the dimension and the lower bound of \(\ric(g)\). Therefore, by \eqref{eq:Harnack}, we have
\begin{align*}
 \int_{B_g(x,1)}u^2\,dV_g
 &\leq
 \frac{C}{\operatorname{Vol}_g(B_g(x,1))}
 \left(
   \int_{B_g(x,1)}u\,dV_g
 \right)^2                                                   \\
 &\leq
 \frac{C}{\operatorname{Vol}_g(B_g(x,1))}
 \left(
   \int_{B_g(x,L)}|\varphi|\,dV_g
 \right)^2                                                   \\
 &\leq
 C\,\frac{\operatorname{Vol}_g(B_g(x,L))}{\operatorname{Vol}_g(B_g(x,1))}
 \int_{B_g(x,L)}|\varphi|^2\,dV_g\\
 & \leq C\int_{B_g(x,L)}|\varphi|^2\,dV_g,
\end{align*}
where the last inequality follows from Bishop--Gromov volume comparison. 

Let \(A_R:=B_g(p,2R)\setminus B_g(p,R)\). Choose a maximal \(1\)-separated set \(\{x_j\}_{j=1}^{N_R}\subset A_R\). Then
\[
A_R\subset\bigcup_j B_g(x_j,1),
\]
while the balls \(B_g(x_j,\tfrac12)\) are pairwise disjoint. For \(y\in B_g(x_j,L)\),
\[
B_g(x_j,\tfrac12)\subset B_g(y,L+\tfrac12)\subset B_g(x_j,2L+\tfrac12).
\]
Hence Bishop--Gromov volume comparison gives
\[
\operatorname{Vol}_g\left( B_g(x_j,\tfrac12)\right)
\geq
C\operatorname{Vol}_g \left(B_g(y,L+\tfrac12)\right).
\]
The family \(\{B_g(x_j,L)\}_j\) has uniformly bounded multiplicity. Indeed, let \(I(y):=\{j;~y \in B_g(x_j,L)\}\). Since the balls \(B_g(x_j,\tfrac12)\) are disjoint, it follows that
\begin{align*}
\operatorname{Vol}_g \left(B_g(y,L+\tfrac12)\right) \ge \sum_{j \in I(y)} \operatorname{Vol}_g \left(B_g(x_j,\tfrac12) \right)
 \ge \#I(y) \cdot C\,\operatorname{Vol}_g \left(B_g(y,L+\tfrac12)\right)
\end{align*}
Therefore, for some \(N\) depending only on the dimension, \(L\) and the lower bound of \(\ric(g)\), we have \(\#I(y)\leq N\). Consequently, 
\begin{align*}
\int_{A_R}u^2\,dV_g
&\leq
\sum_j\int_{B_g(x_j,1)}u^2\,dV_g \\
&\leq
C\sum_j\int_{B_g(x_j,L)}\varphi^2\,dV_g \\
&\leq
CN\int_{B_g(p,2R+L)\setminus B_g(p,R-L)}\varphi^2\,dV_g,
\end{align*}
which proves \eqref{eq:annular-transfer}.

We now prove \eqref{eq:annular-growth-equivalence}. Suppose first
that
\[
\int_{A_R}
|\varphi|^2\,dV_g
=
O\left(e^{2\delta R}\right).
\]
For $R\geq3L$, we have
\[
\left(B_g(p,2R+L)\setminus B_g(p,R-L)\right)
\subset \left(A_{R-L}\cup A_{R+L}\right).
\]
It follows from \eqref{eq:annular-transfer} that
\[
\begin{aligned}
\int_{A_R}
|u|^2\,dV_g
&\leq
C\left(
\int_{A_{R-L}}
|\varphi|^2\,dV_g
+
\int_{A_{R+L}}
|\varphi|^2\,dV_g
\right)\\
&\leq
C'\left(
e^{2\delta(R-L)}
+
e^{2\delta(R+L)}
\right)\\
&\leq
C''e^{2\delta R}.
\end{aligned}
\]
Conversely, Lemma~\ref{lem:vol_compa} and \(\sup_M\varphi=0\) give
\[
u=F-\varphi
=
F+|\varphi|
\geq
|\varphi|,
\]
and hence
\[
\int_{A_R}
|\varphi|^2\,dV_g
\leq
\int_{A_R}
|u|^2\,dV_g.
\]
This proves \eqref{eq:annular-growth-equivalence}.
\end{proof}

\begin{corollary}\label{cor:nonparabolic-rigidity}
Assume the hypotheses of Theorem~\ref{MT1} (ii). Then \(\varphi \equiv 0\).
\end{corollary}

\begin{proof}
We use the notation introduced in the proof of Lemma \ref{lem:transfer-L2}. For $R$ sufficiently large, by Lemma \ref{lem:transfer-L2}, we have
\[
        \int_{A_R}u^2\,dV_g
        \leq C_1e^{2\delta R}.
\]
After increasing $C_1$, the same inequality holds for all $1\leq R$. Since \(0<\delta<2\sqrt{\lambda_1(-\Delta_g)}\), Lemma~\ref{lem:annular-L2-Liouville} gives \(u\equiv0\).
Thus \(F=\varphi\). Lemma~\ref{lem:vol_compa} gives \(F\ge0\), whereas
\(\varphi\le0\). Hence \(F\equiv\varphi\equiv0\).
\end{proof}

%~~~~~~~~~~~~~~~~~~~~~~~~~~~~~~~~~~~~~~~~~~~~~~~~~~~~~~~~~~~~~~~~~~~~~~~~~~~~~~~~~~~~~~~~~
\section{The cscK metric on bounded strictly pseudoconvex domains}\label{sec:pcv}

In this section, we study cscK metrics on bounded strictly pseudoconvex domains in \(\CC^n\). We will also briefly recall some general facts about the Bergman metric on bounded pseudoconvex domains. Note that bounded pseudoconvex domains admit complete K\"ahler--Einstein metrics of negative scalar curvature, see \cite{mokyau}. 
%~~~
\subsection{The Cheng--Yau metric}

Let \(\Omega\subset\CC^n\) be a bounded strictly pseudoconvex domain with \(C^k\)-boundary, where \(k\geq2\). In \cite{Cheng1980OnTE}, Cheng and Yau studied complete K\"ahler metrics arising from defining functions. In the present subsection, we consider metrics of the form
\[
\omega_\rho
:=
-\sqrt{-1}\,\partial\bar\partial\log(-\rho),
\]
where \(\rho\) is a negative globally strictly plurisubharmonic defining function for \(\Omega=\{\rho<0\}\). Here, by globally strictly plurisubharmonic, we mean that \(\rho\) extends as a strictly plurisubharmonic function to a neighborhood of \(\overline{\Omega}\). In particular,
\[
\omega_\rho
=
-\frac{\sqrt{-1}\,\partial\bar\partial\rho}{\rho}
+
\frac{\sqrt{-1}\,\partial\rho\wedge\bar\partial\rho}{\rho^2}
>0,
\]
so \(\omega_\rho\) is a K\"ahler metric on \(\Omega\). Moreover,
the second term shows that the boundary is at infinite
\(\omega_\rho\)-distance, and hence \(\omega_\rho\) is complete.

For a \(C^2\) defining function \(r\), the Fefferman operator is
defined by
\[
J(r)
:=
-\det
\begin{pmatrix}
r & r_{\bar j}\\
r_i & r_{i\bar j}
\end{pmatrix}.
\]
A direct computation gives
\begin{equation}\label{eq:Fefferman-determinant}
\det\bigl((\omega_r)_{i\bar j}\bigr)
=
\frac{J(r)}{(-r)^{n+1}}.
\end{equation}
In particular, whenever \(r\) is smooth in \(\Omega\) and
\(\omega_r\) is K\"ahler, one has \(J(r)>0\) in \(\Omega\), and
\eqref{eq:Fefferman-determinant} yields
\begin{equation}\label{eq:Fefferman-Ricci}
\ric(\omega_r)
=
-(n+1)\omega_r
-
\sqrt{-1}\,\partial\bar\partial\log J(r).
\end{equation}

Cheng and Yau \cite{Cheng1980OnTE} proved that the Dirichlet problem
\begin{equation}\label{eq:Cheng-Yau-Fefferman}
\begin{cases}
J(\rho_0)=1 & \text{in }\Omega,\\
\rho_0=0 & \text{on }\partial\Omega,\\
\end{cases}
\end{equation}
admits a unique solution \(\rho_0\) for which
\[
\omega_0
:=
-\sqrt{-1}\,\partial\bar\partial\log(-\rho_0)
\]
is a complete K\"ahler metric. By \eqref{eq:Fefferman-Ricci}, this metric satisfies
\[
\ric(\omega_0)=-(n+1)\omega_0,
\qquad
R(\omega_0)=-n(n+1).
\]
We refer to \(\rho_0\) as the exact Cheng--Yau defining function and to \(\omega_0\) as the unique K\"ahler--Einstein metric.

For any other \(C^k\)-defining function \(\rho\), there exists a positive function \(u\) such that \(\rho_0=u\rho\). Hence
\[
    J(\rho)=u^{-(n+1)}J(\rho_0)=u^{-(n+1)}
    \qquad\text{on }\partial\Omega.
\]
Multiplication of $\rho$ by a positive constant does not change the associated metric $\omega_\rho$, but it replaces $u$ by a constant multiple of $u$. Thus such a constant rescaling yields $u|_{\partial\Omega}=1$ only when $u|_{\partial\Omega}$ is constant. In what follows, we restrict attention to defining functions satisfying the boundary normalization \(J(\rho)=1\) on \(\partial \Omega\).

\begin{proposition}\label{no cscK by def}
Let $\Omega\subset\mathbb{C}^n$ be a bounded strictly pseudoconvex domain with $C^2$ boundary, and let \(\rho \in C^\infty(\Omega)\cap C^4(\bar{\Omega})\) be a negative strictly plurisubharmonic defining function satisfying \(J(\rho)=1\) on \(\partial\Omega\). If \(\omega:=-\sqrt{-1}\,\partial\bar\partial\log(-\rho)\) has constant scalar curvature, then $\omega$ is K\"ahler--Einstein.
\end{proposition}
\begin{proof}
By the fundamental computations of curvature tensor by Cheng--Yau \cite[(1.22)]{Cheng1980OnTE}, we have
\[
    R_{i\bar j k\bar l}(\omega)
    =
    -\bigl(
        g_{i\bar j}g_{k\bar l}
        +
        g_{i\bar l}g_{k\bar j}
      \bigr)
    +
    O(|\rho|^{-1})
    \qquad\text{as }z\longrightarrow\partial\Omega.
\]
Together with \(g^{i\bar j}=O(|\rho|)\) (\cite[(1.2)]{Cheng1980OnTE}), this gives
\begin{equation}\label{eq:scalar-boundary-asymptotic}
    R(\omega)
    =
    -n(n+1)+O(|\rho|)
    \longrightarrow
    -n(n+1)
    \qquad\text{as }z\longrightarrow\partial\Omega.
\end{equation}

Taking the trace of \eqref{eq:Fefferman-Ricci} with respect to \(\omega\), we obtain
\begin{equation}\label{eq:scalar-Fefferman}
    R(\omega)
    =
    -n(n+1)
    -
    \Delta_\omega\log J(\rho).
\end{equation}
Since \(R(\omega)\) is constant, \eqref{eq:scalar-boundary-asymptotic} implies \(R(\omega)=-n(n+1)\).
It follows from \eqref{eq:scalar-Fefferman} that
\[
    \Delta_\omega\log J(\rho)=0
    \qquad\text{in }\Omega.
\]
Moreover, \(J(\rho)\in C^2(\overline{\Omega})\) is positive in \(\Omega\) and satisfies \(J(\rho)=1\) on \(\partial\Omega\). Hence \(\log J(\rho)\) extends continuously to \(\overline{\Omega}\) and vanishes on \(\partial\Omega\). The maximum principle therefore gives
\[
    \log J(\rho)\equiv0.
\]
Consequently, \eqref{eq:Fefferman-Ricci} implies that \(\omega\) is K\"ahler--Einstein. In particular, since \(\omega\) is complete, the uniqueness argument in \cite[Theorem~4.4]{Cheng1980OnTE} shows that
\(\omega=\omega_0\) where \(\omega_0\) is the complete K\"ahler--Einstein metric constructed by Cheng and Yau.
\end{proof}
%~~~~
\subsection{The Bergman metric}

Let $\Omega\subset\mathbb{C}^n$ be a bounded domain, and let $K_\Omega(z,w)$ denote its Bergman kernel. We write \(K(z):=K_\Omega(z,z)\). The Bergman metric is defined by
\[
    \omega_B
    :=
    \sqrt{-1}\,\partial\bar\partial\log K.
\]
On a bounded strictly pseudoconvex domain with $C^1$ boundary, $\omega_B$ is a complete real-analytic K\"ahler metric \cite{ohsawa1980complete}.

Let \(G(z)\) be the determinant of the Bergman metric. The Bergman invariant function is defined by \(B(z):=\frac{G(z)}{K(z)}\). The transformation laws of \(G\) and \(K\) show that \(B\) is invariant
under biholomorphic maps. Since \(\log G=\log B+\log K\), one has
\begin{equation}\label{eq:Bergman-Ricci}
    \ric(\omega_B)
    +
    \sqrt{-1}\,\partial\bar\partial\log B
    =
    -\omega_B.
\end{equation}

The boundary behavior of the Bergman metric has been widely studied. A key observation is that, on a bounded pseudoconvex domain $\Omega$ with $C^2$ boundary, the Bergman metric $\omega_B$ behaves asymptotically like a Kähler–Einstein metric near a strictly pseudoconvex boundary point \cite{diederich1970randverhalten}. The problem of promoting such boundary asymptotics to a global K\"ahler--Einstein equation belongs to a broader conjectural picture proposed by Yau \cite{yau1992open}. He asked that, for a bounded domain which is not a product domain, the Bergman metric is complete and K\"ahler--Einstein if and only if the domain is homogeneous. For strictly pseudoconvex domains, the Wong--Rosay theorem shows that a bounded homogeneous domain with \(C^2\) boundary must be biholomorphic to the ball \cite{wong1977characterization,rosay1979caracterisation}. A particularly tractable case is therefore Cheng's conjecture: \textit{a smoothly bounded strictly pseudoconvex domain whose Bergman metric is K\"ahler--Einstein must be biholomorphic to the ball}. Fu--Wong and Nemirovski--Shafikov proved this conjecture in complex dimension two \cite{fu1997strictly,nemirovski2006conjectures}, and Huang--Xiao solved it in every complex dimension \(n\geq2\) \cite[Theorem~1]{huang2021bergman}.

This leads naturally to a scalar-curvature version of Cheng's conjecture: \textit{does constant scalar curvature of the Bergman metric already characterize the ball?} The first step is to determine whether the constant-scalar-curvature condition forces the Bergman metric to be K\"ahler--Einstein. The following proposition gives an affirmative answer and, together with Huang--Xiao's theorem, yields the ball characterization for bounded strictly pseudoconvex domains with smooth boundary.

\begin{proposition}\label{prop:Bergman-cscK}
Let $\Omega\subset\mathbb{C}^n$ be a bounded strictly pseudoconvex domain with $C^2$ boundary. If its Bergman metric $\omega_B$ has constant scalar curvature, then $\omega_B$ is K\"ahler--Einstein. Moreover, if  $\partial\Omega\in C^\infty$, then $\Omega$ is biholomorphic to the unit ball.
\end{proposition}
\begin{proof}
By \cite[Corollary~2]{krantz1996bergman}, at every boundary point $p\in\partial\Omega$ one has
\begin{equation}\label{eq:Bergman-boundary-limits}
    \lim_{z\to p}R(\omega_B)(z)=-n,
    \qquad
    \lim_{z\to p}B(z)
    =
    \frac{(n+1)^n\pi^n}{n!}.
\end{equation}
Since $R(\omega_B)$ is constant, the first limit gives \(R(\omega_B)\equiv -n\). Taking the trace of \eqref{eq:Bergman-Ricci} with respect to $\omega_B$, we obtain
\[
    R(\omega_B)
    +
    \Delta_{\omega_B}\log B
    =
    -n.
\]
It follows that
\[
    \Delta_{\omega_B}\log B=0
    \qquad\text{in }\Omega.
\]
The second limit in \eqref{eq:Bergman-boundary-limits} shows that $\log B$ extends continuously to $\overline{\Omega}$ with a constant boundary value. By the weak maximum principle, we conclude
that \(\log B\) is constant on \(\bar\Omega\). Equation \eqref{eq:Bergman-Ricci} therefore implies that \(\omega_B\) is K\"ahler--Einstein. If $\partial\Omega\in C^\infty$, the final assertion follows from \cite[Theorem~1]{huang2021bergman}.
\end{proof}

%~~~~~~~~~~~~~~~~~~~~~~~~~~~~~~~~~~~~~~~~~~~~~~~~~~~~~~~~~~~~~~~~~~~~~~~~~~~~~~~~~~~~~~~~~~~~~~~~~~~~~~~~~~~~~~~~~~~~~~~~~~~~~~~~~~~~~~~~~~~~~~~~~~~~~~~~~~~~~~~

\bibliographystyle{amsalpha}
\bibliography{wpref}
\end{document}